\documentclass[11pt,reqno]{article}

\usepackage[usenames]{color}
\usepackage{amssymb}
\usepackage{amsmath}
\usepackage{cases}
\usepackage{amsthm}
\usepackage{amsfonts}
\usepackage{amscd}
\usepackage{graphicx}

\usepackage[colorlinks=true,
linkcolor=webgreen,
filecolor=webbrown,
citecolor=webgreen]{hyperref}

\definecolor{webgreen}{rgb}{0,.5,0}
\definecolor{webbrown}{rgb}{.6,0,0}
\usepackage{color}
\usepackage{fullpage}
\usepackage{float}

\usepackage{cases}

\usepackage{graphics,amsmath,amssymb}
\usepackage{amsthm}
\usepackage{amsfonts}
\usepackage{latexsym}

\begin{document}

\theoremstyle{plain}
\newtheorem{theorem}{Theorem}
\newtheorem{remark}{Remark}
\newtheorem{lemma}{Lemma}

\newtheorem*{proposition}{Proposition}

\begin{center}
\vskip 1cm{\Large\bf A determinantal formula for the hyper-sums \\
\vskip .07in of powers of integers}
\vskip .2in \large Jos\'{e} Luis Cereceda \\
{\normalsize Collado Villalba, 28400 (Madrid), Spain} \\
\href{mailto:jl.cereceda@movistar.es}{\normalsize{\tt jl.cereceda@movistar.es}}
\end{center}

\begin{abstract}
For non-negative integers $r$ and $m$, let $S_m^{(r)}(n)$ denote the $r$-fold summation (or hyper-sum) over the first $n$ positive integers to the $m$th powers, with the initial condition $S_m^{(0)}(n) =n^m$. In this paper, we derive a new determinantal formula for $S_m^{(r)}(n)$. Specifically, we show that, for all integers $r\geq 0$ and $m \geq 1$, $S_m^{(r)}(n)$ is proportional to $S_1^{(r)}(n)$ times the determinant of a lower Hessenberg matrix of order $m-1$ involving the Bernoulli numbers and the variable $N_r = n + \frac{r}{2}$. Furthermore, whenever $r\geq 1$, evaluating this determinant gives us $S_m^{(r)}(n)$ as $S_1^{(r)}(n)$ times an even or odd polynomial in $N_r$ of degree $m-1$, depending on whether $m$ is odd or even.
\end{abstract}

\section{Introduction}

For integers $m,r \geq 0$ and $n \geq 1$, the hyper-sums of powers of integers $S_m^{(r)}(n)$ are defined recursively as
\begin{equation*}
S_m^{(r)}(n) = \left\{
              \begin{array}{ll}
                n^m  & \text{if}\,\, r =0, \\
                \sum_{i=1}^n S_m^{(r-1)}(i) &  \text{if}\,\, r \geq 1.
              \end{array}
            \right.
\end{equation*}
In particular, $S_m^{(1)}(n) = \sum_{i=1}^{n} S_m^{(0)}(i)$ is the ordinary sum of powers of integers $S_m(n) = 1^m + 2^m + \cdots + n^m$. Furthermore, we have \cite[p. 281]{knuth}
\begin{equation}\label{knuth}
  S_1^{(r)}(n) = \binom{n+r}{r+1}, \quad\text{and}\quad S_2^{(r)}(n) = \frac{2n+r}{r+2} S_1^{(r)}(n).
\end{equation}

For its intrinsic interest, next we give an explicit representation of $S_m^{(r)}(n)$ as a polynomial in $n$ of degree $m+r$ without constant term. To this end, we first recall that the hyper-sums $S_m^{(r+1)}(n)$, $S_m^{(r)}(n)$, and $S_{m+1}^{(r)}(n)$ satisfy the recurrence relation \cite[Theorem 1]{cere}
\begin{equation}\label{rec}
  S_m^{(r+1)}(n) = \frac{n+r}{r} S_m^{(r)}(n) - \frac{1}{r} S_{m+1}^{(r)}(n),
\end{equation}
which applies to any integers $m \geq 0$ and $r \geq 1$. (For the sake of completeness, we provide a proof of the recurrence \eqref{rec} in the appendix.) Solving this recurrence gives us \cite[Theorem 3]{cere}
\begin{equation}\label{rec2}
   S_m^{(r+1)}(n) = \frac{1}{r!} \sum_{i=0}^r (-1)^i q_{r,i}(n) S_{m+i}(n),
\end{equation}
where $q_{r,i}(n)$ is the following polynomial in $n$ of degree $r-i$:
\begin{equation}\label{poly}
  q_{r,i}(n) = \sum_{j=0}^{r-i} \binom{i+j}{i} \genfrac{[}{]}{0pt}{}{r+1}{i+j+1} n^j,
\end{equation}
and the $\genfrac{[}{]}{0pt}{}{m}{n}$'s are the (unsigned) Stirling numbers of the first kind. (Note that \eqref{rec2} holds for $r=0$ if we set $q_{0,0}(n) =1$. Let us also observe that $q_{r,i}(n)$ is equal to $\genfrac{[}{]}{0pt}{}{r+n+1}{i+n+1}_{n+1}$, where $\genfrac{[}{]}{0pt}{}{m}{n}_{r}$ denotes the $r$-Stirling numbers of the first kind \cite{broder}. The formula $S_m^{(r+1)}(n) = \frac{1}{r!} \sum_{i=0}^r (-1)^i \genfrac{[}{]}{0pt}{}{r+n+1}{i+n+1}_{n+1} S_{m+i}(n)$ has been derived recently and independently by Karg{\i}n {\it et al.\/} \cite[Equation (4.7)]{kargin}.)

On the other hand, as is well-known, the power sum polynomial $S_{m+i}(n)$ can be expressed by means of the Bernoulli formula
\begin{equation}\label{ber}
  S_{m+i}(n) = \frac{1}{m+i+1} \sum_{t=1}^{m+i+1} (-1)^{m+i+1-t} \binom{m+i+1}{t} B_{m+i+1-t} \, n^t,
\end{equation}
where the $B_j$'s are the Bernoulli numbers $B_0 =1$, $B_1 = -\frac{1}{2}$, $B_2 = \frac{1}{6}$, $B_3 =0$, etc. Thus, by substituting \eqref{poly} and \eqref{ber} into \eqref{rec2}, it can be shown that
\begin{equation}\label{polyn}
  S_m^{(r)}(n) = \sum_{k=1}^{m+r} c_{m,r}^{k} n^k,
\end{equation}
where the coefficient $c_{m,r}^{k}$ is given by
\begin{equation*}
  c_{m,r}^{k} = \frac{(-1)^{m+1-k}}{(r-1)!} \sum_{i=0}^{r-1} \sum_{j=0}^{k-1} \frac{(-1)^j}{m+i+1} \binom{i+j}{i}
  \binom{m+i+1}{k-j} \genfrac{[}{]}{0pt}{}{r}{i+j+1} B_{m+i+j+1-k},
\end{equation*}
for integers $m \geq 0$, $r \geq 1$, and $k =1,2,\ldots, m+r$, and where $B_t =0$ whenever $t$ is a negative integer. In particular, we find that
\begin{equation*}
  c_{m,r}^{1} = \frac{(-1)^{m}}{(r-1)!} \sum_{i=0}^{r-1} \genfrac{[}{]}{0pt}{}{r}{i+1} B_{m+i}.
\end{equation*}

\begin{remark}
From \eqref{rec}, it is immediate to derive the following recurrence relation between the coefficients of the hyper-sums polynomials $S_m^{(r+1)}(n)$, $S_m^{(r)}(n)$, and $S_{m+1}^{(r)}(n)$, namely,
\begin{equation*}
  c_{m,r+1}^k = c_{m,r}^k + \frac{1}{r} \big( c_{m,r}^{k-1} - c_{m+1,r}^k \big),
\end{equation*}
with $m \geq 0$, $r \geq 1$, and $k =1,2,\ldots,m+r+1$, and where it is understood that $c_{m,r}^0 = c_{m,r}^{m+r+1} =0$.
\end{remark}

Since the work of the German mathematician Johann Faulhaber (1580-1635), it is known that $S_{2m-1}^{(r)}(n)$ [respectively, $S_{2m}^{(r)}(n)$] can alternatively be expressed as $S_{1}^{(r)}(n)$ [$S_{2}^{(r)}(n)$] times a polynomial in $n(n+r)$. This result was rigorously proved by Knuth (see \cite[Theorem, p. 280]{knuth}), and can be established as the following theorem.

\begin{theorem}[Faulhaber-Knuth]\label{th:1}
  For all positive integers $r$ and $m$, there exist polynomials $F_{2m-1}^{(r)}$ and $F_{2m}^{(r)}$ in $n(n+r)$ of degree $m-1$ such that
  \begin{equation*}
    S_{2m-1}^{(r)}(n) = S_{1}^{(r)}(n) \, F_{2m-1}^{(r)} \big(n(n+r) \big), \quad\text{and}\quad
    S_{2m}^{(r)}(n) = S_{2}^{(r)}(n) \, F_{2m}^{(r)} \big(n(n+r) \big).
  \end{equation*}
\end{theorem}

It is important to notice that, because of the relation $n(n+r) = \big( n +\frac{r}{2} \big)^2 - \frac{1}{4}r^2$, the above theorem can equivalently be stated as follows.

\begin{theorem}[Faulhaber-Knuth]\label{th:2}
  For all positive integers $r$ and $m$, there exist even polynomials $K_{2m-1}^{(r)}$ and $K_{2m}^{(r)}$ in $N_r = n + \frac{r}{2}$ of degree $2m-2$ such that
  \begin{equation*}
    S_{2m-1}^{(r)}(n) = S_{1}^{(r)}(n) \, K_{2m-1}^{(r)}(N_r), \quad\text{and}\quad
    S_{2m}^{(r)}(n) = S_{2}^{(r)}(n) \, K_{2m}^{(r)}(N_r).
  \end{equation*}
\end{theorem}

\begin{remark}\label{rk:1}
Since $S_{2}^{(r)}(n)$ is proportional to $N_r  S_{1}^{(r)}(n)$, from Theorem \ref{th:2} it follows that $S_{2m}^{(r)}(n)$ can equally be expressed as $S_{1}^{(r)}(n)$ times an odd polynomial in $N_r$ of degree $2m-1$.
\end{remark}

In view of Theorem \ref{th:2} and Remark \ref{rk:1}, the Faulhaber-Knuth result can be formulated in another equivalent way, as stated in the next theorem.

\begin{theorem}[Faulhaber-Knuth]\label{th:3}
  For all positive integers $r$ and $m$, there exist polynomials $G_{m}^{(r)}$ in $N_r = n + \frac{r}{2}$ of degree $m-1$ such that
  \begin{equation}\label{polg}
    S_{m}^{(r)}(n) = S_{1}^{(r)}(n) \, G_{m}^{(r)}(N_r),
  \end{equation}
  where the $G_{m}^{(r)}$'s are even or odd, according as $m$ is odd or even.
\end{theorem}

In this paper, we give a determinantal version of Theorem \ref{th:3}. More precisely, in section 2, we show that, for all integers $r \geq 0$ and $m\geq 1$, $S_{m}^{(r)}(n)$ can be expressed (leaving aside a factor involving $m$ and $r$) as $S_{1}^{(r)}(n)$ times a determinant of order $m-1$ depending on $N_r$ (see Theorem \ref{th:4} below). In other words, we provide an explicit determinantal formula for the polynomials $G_m^{(r)}(N_r)$. Furthermore, in section 3, we offer a new proof of the fact that, whenever $r \geq 1$, the polynomials $G_m^{(r)}(N_r)$ are even or odd depending on the parity of $m$ (see Theorem \ref{th:5} below).

\section{A determinantal formula for $S_{m}^{(r)}(n)$}

Next, we establish the following theorem which constitutes the main result of this paper.

\begin{theorem}\label{th:4}
  Let $N_r = n + \frac{r}{2}$. Then, for all integers $r \geq 0$ and $m \geq 1$, $S_{m}^{(r)}(n)$ can be expressed as
  \begin{equation}\label{det}
    S_{m}^{(r)}(n) = S_{1}^{(r)}(n) \frac{(-1)^{m-1}}{(r+2)^{\overline{m-1}}} \det{H_m^{(r)}(N_r)},
  \end{equation}
  where $r^{\overline{m}}= r(r+1)\ldots (r+m-1)$ and $r^{\overline{0}} =1$, and where $H_m^{(r)}(N_r)$ is the following lower Hessenberg matrix of order $m-1$:
  \begin{equation}\label{det1}
  H_m^{(r)}(N_r) =
  \begin{pmatrix}
  -2N_r & \!\! r+2 & \!\! 0 & \!\! \hdots & \!\! \hdots & \!\! 0 \\[3pt]
  r\binom{3}{1}B_2 & \!\! -3N_r & \!\! r+3 & \!\! 0 & \!\! \hdots & \!\! 0 \\[3pt]
  r\binom{4}{1}B_3 & \!\! r\binom{4}{2}B_2 & \!\! -4N_r & \!\! r+4 & \!\! \ddots & \!\! 0  \\[3pt]
  \vdots & \!\! \vdots & \!\! \ddots  & \!\! \ddots & \!\! \ddots & \!\! 0 \\[3pt]
  r\binom{m-1}{1}B_{m-2} & \!\! r\binom{m-1}{2}B_{m-3} & \!\! r\binom{m-1}{3}B_{m-4} & \!\! \hdots & \!\! -(m-1)N_r &
   \!\! r+m-1 \\[5pt]
  r\binom{m}{1}B_{m-1} & \!\! r\binom{m}{2}B_{m-2} & \!\! r\binom{m}{3}B_{m-3} & \!\! \hdots &  \! r\binom{m}{m-2}B_2 &
   \!\! -mN_r
\end{pmatrix},
\end{equation}
  with the $B_j$'s denoting the Bernoulli numbers.
\end{theorem}

\begin{remark}
It is understood that $\det{H_1^{(r)}(N_r)} =1$, so that \eqref{det} holds trivially for $m=1$. On the other hand, it is easily seen that $\det{H_m^{(0)}(N_0)} = (-1)^{m-1} 2^{\overline{m-1}} n^{m-1}$, and then \eqref{det} gives us $S_{m}^{(0)}(n) = n^m$, as it should. Furthermore, from \eqref{det} it follows that $S_{2}^{(r)}(n) = \frac{2n+r}{r+2} S_1^{(r)}(n)$, in accordance with \eqref{knuth}.
\end{remark}

\begin{remark}
From \eqref{polg} and \eqref{det}, it follows immediately that, for all integers $r \geq 0$ and $m\geq 1$, the polynomials $G_{m}^{(r)}(N_r)$ are given by the determinantal formula
\begin{equation}\label{detg}
  G_{m}^{(r)}(N_r) = \frac{(-1)^{m-1}}{(r+2)^{\overline{m-1}}} \det{H_m^{(r)}(N_r)},
\end{equation}
where $H_m^{(r)}(N_r)$ is the matrix \eqref{det1}, and $G_1^{(r)}(N_r) =1$.
\end{remark}

The proof of Theorem \ref{th:4} is based on the following lemma.

\begin{lemma}
Let $N_r = n + \frac{r}{2}$. Then, for all integers $r \geq 0$ and $m \geq 2$, the following recurrence relation holds true
  \begin{equation}\label{lm:1}
    (m+r) S_{m}^{(r)}(n) = m N_r S_{m-1}^{(r)}(n) - r \sum_{k=1}^{m-2} \binom{m}{k} B_{m-k} S_{k}^{(r)}(n),
  \end{equation}
where the summation on the right-hand side is zero if $m=2$.
\end{lemma}
\begin{proof}
We start with the Bernoulli formula for the power sum polynomial $S_{m-1}(i) = \sum_{j=1}^i j^{m-1}$, namely,
\begin{equation*}
  S_{m-1}(i) = \frac{1}{m} i^m + \frac{1}{2} i^{m-1} + \frac{1}{m} \sum_{k=1}^{m-2} \binom{m}{k} B_{m-k} i^k,
\end{equation*}
which applies to $m \geq 2$. Therefore, it follows that
\begin{equation*}
  m S_{m-1}(i) = i^m + \frac{1}{2} m \, i^{m-1} + \sum_{k=1}^{m-2} \binom{m}{k} B_{m-k} i^k.
\end{equation*}
Summing from $i=1$ to $n$ on both sides of this equation yields
\begin{equation*}
  m S_{m-1}^{(2)}(n) = S_m(n) + \frac{1}{2} m S_{m-1}(n) + \sum_{k=1}^{m-2} \binom{m}{k} B_{m-k} S_k(n).
\end{equation*}
Clearly, by iterating the procedure $r$ times, we get
\begin{equation}\label{lm:2}
  m S_{m-1}^{(r+1)}(n) = S_m^{(r)}(n) + \frac{1}{2} m S_{m-1}^{(r)}(n) + \sum_{k=1}^{m-2} \binom{m}{k} B_{m-k}
  S_k^{(r)}(n).
\end{equation}
Now, from \eqref{rec} we have
\begin{equation}\label{lm:3}
  S_{m-1}^{(r+1)}(n) = \frac{n+r}{r} S_{m-1}^{(r)}(n) - \frac{1}{r} S_{m}^{(r)}(n),
\end{equation}
and hence, using \eqref{lm:3} in \eqref{lm:2}, we quickly obtain \eqref{lm:1}.
\end{proof}

For the proof of Theorem \ref{th:4}, rewrite \eqref{lm:1} in the form
\begin{equation}\label{lm:4}
    (r+j) S_{j}^{(r)}(n) = j N_r S_{j-1}^{(r)}(n) - r \sum_{k=1}^{j-2} \binom{j}{k} B_{j-k} S_{k}^{(r)}(n),
\end{equation}
for integers $r \geq 0$ and $j \geq 2$. Thus, letting successively $j =2,3,4\ldots,m$ in \eqref{lm:4} gives rise to the following system of equations in the unknowns $S_1^{(r)}(n), S_2^{(r)}(n)$, \ldots, $S_m^{(r)}(n)$:
\begin{equation*}
\begin{aligned}
-2& N_r S_1^{(r)}(n) + (r+2) S_2^{(r)}(n) =0, \\
& \,\,\,\, r\binom{3}{1}B_2 S_1^{(r)}(n) -3N_r S_2^{(r)}(n) + (r+3)S_3^{(r)}(n) = 0,  \\
& \,\,\,\, r\binom{4}{1}B_3 S_1^{(r)}(n) + r\binom{4}{2}B_2 S_2^{(r)}(n) -4N_r S_3^{(r)}(n) + (r+4)S_4^{(r)}(n) =0, \\
&  \,\,\,\,\,\, \quad \vdots \\
& \,\,\,\, r\binom{m}{1}B_{m-1} S_1^{(r)}(n) + r\binom{m}{2}B_{m-2} S_2^{(r)}(n) + \cdots + r\binom{m}{m-2}B_2 S_{m-2}^{(r)}(n) \\
& \quad\quad\quad\quad\quad\quad\quad\quad\quad\quad\quad\quad\quad\quad\quad\quad\quad\quad\quad\quad\quad \,\,\,\,\,
-m N_r S_{m-1}^{(r)}(n) + (r+m) S_{m}^{(r)}(n)=0.
\end{aligned}
\end{equation*}
This system of equations augmented with the trivial relationship $(r+1)S_1^{(r)}(n) = (r+1)S_1^{(r)}(n)$ can be written in matrix form as
\begin{equation*}
\!
\begin{pmatrix} r+1 & 0 & 0 & \! \hdots & 0 & 0 \\[3pt]
-2N_r & r+2 & 0 & \! \hdots & 0 & 0 \\[3pt]
r\binom{3}{1}B_2 & -3N_r & r+3 & \! \hdots  & 0 & 0 \\[3pt]
\vdots & \vdots & \ddots & \! \ddots & \vdots & \vdots \\[3pt]
r\binom{m-1}{1}B_{m-2} & \! r\binom{m-1}{2}B_{m-3} & \! r\binom{m-1}{3}B_{m-4} & \!\hdots & \! r+m-1 & 0\\[5pt]
r\binom{m}{1}B_{m-1} & r\binom{m}{2}B_{m-2} & r\binom{m}{3}B_{m-3} & \! \hdots & -m N_r & r+m
\end{pmatrix}\!
\begin{pmatrix} S_1^{(r)}(n) \\[2pt] S_2^{(r)}(n) \\[2pt] S_3^{(r)}(n) \\ \vdots \\
S_{m-1}^{(r)}(n) \\[1pt] S_m^{(r)}(n)
\end{pmatrix}\!
= \! \begin{pmatrix}
(r+1)S_1^{(r)}(n) \\[4pt] 0 \\[4pt] 0 \\ \vdots \\[3pt] 0 \\[3pt] 0
\end{pmatrix}.
\end{equation*}
Thus, solving for $S_m^{(r)}(n)$ and applying Cramer's rule to the above $m \times m$ lower triangular system of linear
equations yields
\begin{equation*}
S_m^{(r)}(n)= \frac{1}{(r+1)^{\overline{m}}}
\begin{vmatrix}
r+1 &  0 & 0 & \! \hdots &  0 & \!\!\!\!(r+1)S_1^{(r)}(n) \\[3pt]
-2N_r & r+2 & 0 & \!\hdots &  0 & \!\! 0 \\[3pt]
r\binom{3}{1}B_2 & -3N_r & r+3 & \!\hdots &  0 & \!\! 0 \\[3pt]
\vdots & \vdots & \ddots  & \! \ddots & \vdots & \!\! \vdots \\[3pt]
r\binom{m-1}{1}B_{m-2} & \! r\binom{m-1}{2}B_{m-3} & \! r\binom{m-1}{3}B_{m-4} & \! \hdots & \! r+m-1 & \!\! 0\\[5pt]
r\binom{m}{1}B_{m-1} & r\binom{m}{2}B_{m-2} & r\binom{m}{3}B_{m-3} & \! \hdots & -mN_r & \!\! 0
\end{vmatrix}.
\end{equation*}
Finally, by expanding the above determinant with respect to the last column, we get \eqref{det}.

As a concrete example, next we quote the results obtained from \eqref{detg} for $G_{5}^{(7)}(N_7)$ and $G_{6}^{(7)}(N_7)$, namely,
\begin{align*}
G_5^{(7)}(N_7) & = \frac{1}{9^{\overline{4}}}
\begin{vmatrix}
-2N_7 & 9 & 0 & 0 \\[3pt]
\frac{7}{2} & -3N_7 & 10 & 0 \\[3pt]
0 & 7 & -4N_7 & 11 \\[3pt]
-\frac{7}{6} & 0 & \frac{35}{3} & -5N_7 \\[3pt]
\end{vmatrix} = \frac{N_7^4}{99} - \frac{35N_7^2}{198} + \frac{7}{16}, \\[-3mm]
\intertext{and}
G_6^{(7)}(N_7) & = -\frac{1}{9^{\overline{5}}}
\begin{vmatrix}
-2N_7 & 9 & 0 & 0 & 0 \\[3pt]
\frac{7}{2} & -3N_7 & 10 & 0 & 0 \\[3pt]
0 & 7 & -4N_7 & 11 & 0 \\[3pt]
-\frac{7}{6} & 0 & \frac{35}{3} & -5N_7 & 12 \\[3pt]
0 & -\frac{7}{2} & 0 & \frac{35}{2} & -6N_7 \\[3pt]
\end{vmatrix} = \frac{2N_7^5}{429} - \frac{49 N_7^3}{429} + \frac{6419 N_7}{10296},
\end{align*}
from which it follows that
\begin{align*}
  S_5^{(7)}(n) & = \frac{1}{1584}\binom{n+7}{8} \left[ 16 \left(n+ \frac{7}{2}\right)^4 - 280 \left(n+ \frac{7}{2}\right)^2
  + 693 \right], \\[-3mm]
\intertext{and}
  S_6^{(7)}(n) & = \frac{1}{10296}\binom{n+7}{8} \left[ 48 \left(n+ \frac{7}{2}\right)^5 - 1176 \left(n+ \frac{7}{2}\right)^3
  + 6419 \left(n+ \frac{7}{2}\right)\right],
\end{align*}
respectively.

As another simple example, let us observe that, for $m=3$, equation \eqref{det} gives us
\begin{equation*}
  S_3^{(r)}(n) = \frac{S_1^{(r)}(n)}{(r+2)(r+3)}
  \begin{vmatrix}
-2N_r & r+2 \\[3pt]
\frac{1}{2}r & -3N_r \\[3pt]
\end{vmatrix}
=  \binom{n+r}{r+1} \frac{6n^2 +6rn + r(r-1)}{(r+2)(r+3)}.
\end{equation*}
Of course, when $r=1$, the last formula reduces to the well-known identity
\begin{equation*}
  S_3(n) = 1^3 + 2^3 + \cdots + n^3 = \binom{n+1}{2}^2.
\end{equation*}

\begin{remark}
Clearly, the term of maximum degree in $N_r$ of the polynomial $G_m^{(r)}(N_r)$ corresponds to the product of the entries on the main diagonal of \eqref{det1}. Thus, using \eqref{det}, and recalling that $N_r = n + \frac{r}{2}$ and $S_1^{(r)}(n) = \binom{n+r}{r+1}$, it is easily seen that the leading coefficient of the hyper-sum polynomial \eqref{polyn} is given by $c_{m,r}^{m+r} = \frac{m!}{(m+r)!}$, which holds for any integers $m \geq 0$ and $r \geq 1$. In particular, for $r=1$, we retrieve the well-known result $c_{m,1}^{m+1} = \frac{1}{m+1}$.
\end{remark}

\section{A further characterization of the polynomials $G_{m}^{(r)}(N_r)$}

For $r=0$, the polynomial $G_{m}^{(r)}(N_r)$ is equal to $G_{m}^{(0)}(N_0) = N_0^{m-1}$ for any integer $m \geq 1$. However, whenever $r \geq 1$, $G_{m}^{(r)}(N_r)$ involves powers of $N_r$ other than $m-1$. In this section, we show that, for all integers $r \geq 1$ and $m \geq 1$, the polynomials $G_{m}^{(r)}(N_r)$ are necessarily of the form given by the following theorem.

\begin{theorem}\label{th:5}
Let $N_r = n + \frac{r}{2}$ and $S_{m}^{(r)}(n) = S_{1}^{(r)}(n)\, G_{m}^{(r)}(N_r)$. Then, for all integers $r \geq 1$ and $m \geq 1$, we have
  \begin{align}
  G_{2m-1}^{(r)}(N_r) & = \sum_{j=0}^{m-1} g_{2m-1,j}^{(r)} \, N_r^{2j}, \label{polgo} \\[-5mm]
  \intertext{and}
   G_{2m}^{(r)}(N_r) & = \sum_{j=0}^{m-1} g_{2m,j}^{(r)} \, N_r^{2j+1}, \label{polge}
  \end{align}
for certain non-zero (rational) coefficients $g_{2m-1,j}^{(r)}$ and $g_{2m,j}^{(r)}$, $j =0,1,\ldots, m-1$. Furthermore, the coefficients $\{ g_{2m-1,m-1}^{(r)}, g_{2m-1,m-2}^{(r)}, \ldots, g_{2m-1,0}^{(r)} \}$ have alternating signs, with the sign of the leading coefficient $g_{2m-1,m-1}^{(r)}$ being positive. The same happens with the coefficients in the set $\{ g_{2m,m-1}^{(r)}, g_{2m,m-2}^{(r)}, \ldots, g_{2m,0}^{(r)} \}$.
\end{theorem}
\begin{proof}
The theorem can be easily proved by means of the recurrence \eqref{lm:1} and the use of complete mathematical induction. To this end, we need the following couple of well-known properties of the Bernoulli numbers:
\begin{description}
\item[Property 1]: $B_{2j+1} =0$, for all $j =1,2,\ldots \, $,
\item[Property 2]: $\text{sign}\,  B_{2j} = (-1)^{j+1}$, for all $j =1,2,\ldots \,$.
\end{description}

First we show the validity of \eqref{polgo}.  Putting $S_{2m-1}^{(r)}(n) = S_1^{(r)}(n)\, G_{2m-1}^{(r)}(N_r)$ in \eqref{lm:1}, and factoring out the term $S_1^{(r)}(n)$, we obtain
\begin{equation*}
    (2m-1+r) G_{2m-1}^{(r)}(N_r) = (2m-1) N_r G_{2m-2}^{(r)}(N_r) - r \sum_{k=1}^{2m-3} \binom{2m-1}{k} B_{2m-k-1} G_{k}^{(r)}(N_r).
\end{equation*}
By Property 1, this can be expressed equivalently as
\begin{equation}\label{rec3}
    (2m-1+r) G_{2m-1}^{(r)}(N_r) = (2m-1) N_r G_{2m-2}^{(r)}(N_r) - r \sum_{k=1}^{m-1} \binom{2m-1}{2k-1} B_{2m-2k} G_{2k-1}^{(r)}(N_r),
\end{equation}
which applies to any integers $r\geq 0$ and $m \geq 2$. When $r=0$, \eqref{rec3} reduces to $G_{2m-1}^{(0)}(N_0) = N_0 G_{2m-2}^{(0)}(N_0)$, and then, assuming that $G_{2m-2}^{(0)}(N_0) = N_0^{2m-3}$, we get $G_{2m-1}^{(0)}(N_0) = N_0^{2m-2}$. Next, we consider the non-trivial case where $r$ is any arbitrary integer $r \geq 1$. Thus, the recurrence \eqref{rec3} gives us $G_{2m-1}^{(r)}(N_r)$ in terms of the lower-degree polynomials $G_{1}^{(r)}(N_r), G_{3}^{(r)}(N_r), \ldots, G_{2m-3}^{(r)}(N_r)$, and $G_{2m-2}^{(r)}(N_r)$. As the complete inductive hypothesis, we assume that $G_{2k-1}^{(r)}(N_r)$ is of the form \eqref{polgo} for all $k =1,2,\ldots,m-1$. Also, we assume that $G_{2k}^{(r)}(N_r)$ is of the form \eqref{polge} for $k =m-1$. (Please note that, when we say that a polynomial ``is of the form'' \eqref{polgo} or \eqref{polge}, it is understood that the corresponding coefficients are such that they satisfy the associated properties stipulated after equation \eqref{polge}.) Therefore, we have
\begin{align*}
    (2m-1+r) G_{2m-1}^{(r)}(N_r) & = (2m-1) N_r \sum_{j=0}^{m-2} g_{2m-2,j}^{(r)} \, N_r^{2j+1} \\[-2mm]
    & \quad\quad -r \sum_{k=1}^{m-1} \binom{2m-1}{2k-1} B_{2m-2k} \sum_{j=0}^{k-1} g_{2k-1,j}^{(r)} \, N_r^{2j} \\[-2mm]
    & = (2m-1) g_{2m-2,m-2}^{(r)} N_r^{2m-2} + (2m-1) \sum_{j=1}^{m-2} g_{2m-2,j-1}^{(r)} \, N_r^{2j} \\[-1mm]
    &  \quad\quad -r \sum_{j=0}^{m-2} \left[ \sum_{k=j+1}^{m-1} \binom{2m-1}{2k-1} B_{2m-2k} \, g_{2k-1,j}^{(r)}
    \right] N_r^{2j},
\end{align*}
from which we deduce that $G_{2m-1}^{(r)}(N_r) = \sum_{j=0}^{m-1} g_{2m-1,j}^{(r)} \, N_r^{2j}$, where
\vspace{2mm}
  \begin{subnumcases}{g_{2m-1,j}^{(r)} =}
    \displaystyle \, \frac{2m-1}{2m-1+r} \, g_{2m-2,m-2}^{(r)}, \,\,\, \text{for}
     \,\, j = m-1; \label{case1} \\[1mm]
    \displaystyle \, \frac{1}{2m-1+r} \bigg[(2m-1) g_{2m-2,j-1}^{(r)}
     -r \sum_{k=j+1}^{m-1} \binom{2m-1}{2k-1} B_{2m-2k} \, g_{2k-1,j}^{(r)} \bigg], \label{case2} \\[-1mm]
       \qquad\qquad\qquad\qquad\qquad\qquad\qquad\qquad\qquad\qquad\qquad\quad \text{for} \,\, j = 1,2,\ldots,
     m-2; \notag \\
    \displaystyle \, -\frac{r}{2m-1+r} \sum_{k=1}^{m-1} \binom{2m-1}{2k-1} B_{2m-2k} \, g_{2k-1,0}^{(r)},
       \,\,\, \text{for} \,\, j=0. \label{case3}
  \end{subnumcases}
Examining the three-case equation above, we first note that the coefficient $g_{2m-2,m-2}^{(r)}$ in \eqref{case1} corresponds to the leading coefficient of $G_{2m-2}^{(r)}(N_r)$ which, by the induction hypothesis, is a positive (rational) number. Hence, so is the leading coefficient $g_{2m-1,m-1}^{(r)}$ of $G_{2m-1}^{(r)}(N_r)$. Likewise, by the induction hypothesis, the coefficient $g_{2m-2,j-1}^{(r)}$ appearing in \eqref{case2} is a non-zero (rational) number with sign $(-1)^{m-1-j}$. Furthermore, by invoking Property 2, it can be seen that the sign of the (non-zero) term $B_{2m-2k} \, g_{2k-1,j}^{(r)}$ is equal to $(-1)^{m-k+1} (-1)^{k-1-j} = (-1)^{m-j}$, which does not depend on $k$. Therefore, the overall expression in \eqref{case2} turns out to be a non-zero (rational) number with sign $(-1)^{m-1-j}$. Similarly, the sign of the (non-zero) term $B_{2m-2k} \, g_{2k-1,0}^{(r)}$ in \eqref{case3} is equal to $(-1)^{m}$, and thus the sign of $g_{2m-1,0}^{(r)}$ is $(-1)^{m-1}$. Putting all these things together, it follows that, for each $j=0,1,\ldots,m-1$, the coefficient $g_{2m-1,j}^{(r)}$ is a non-zero (rational) number with sign $(-1)^{m-1-j}$. On the other hand,  it is obvious that, for $m=1$, $G_{2m-1}^{(r)}(N_r)$ is of the form \eqref{polgo} since $G_1^{(r)}(N_r) = 1$. Hence, we conclude that, for all integers $r \geq 1$ and $m \geq 1$, the polynomial $G_{2m-1}^{(r)}(N_r)$ is of the form \eqref{polgo}.

The validity of \eqref{polge} can be proved analogously. Now, the counterpart of \eqref{rec3} reads
\begin{equation}\label{rec4}
  (2m+r) G_{2m}^{(r)}(N_r) = 2m N_r G_{2m-1}^{(r)}(N_r) - r \sum_{k=1}^{m-1} \binom{2m}{2k} B_{2m-2k} G_{2k}^{(r)}(N_r),
\end{equation}
which applies to any integers $r\geq 0$ and $m \geq 1$. As before, when $r=0$, from \eqref{rec4} we have $G_{2m}^{(0)}(N_0) = N_0^{2m-1}$, provided that $G_{2m-1}^{(0)}(N_0) = N_0^{2m-2}$. Furthermore, when $m=1$, from \eqref{rec4} we obtain $G_{2}^{(r)}(N_r) = \frac{2}{r+2} N_r$, which is of the form \eqref{polge}. Considering the non-trivial case where $r$ and $m$ are arbitrary integers $r \geq 1$ and $m \geq 2$, the recurrence \eqref{rec4} gives us $G_{2m}^{(r)}(N_r)$ in terms of the lower-degree polynomials $G_{2}^{(r)}(N_r), G_{4}^{(r)}(N_r), \ldots, G_{2m-2}^{(r)}(N_r)$, and $G_{2m-1}^{(r)}(N_r)$. Thus, assuming that $G_{2k}^{(r)}(N_r)$ is of the form \eqref{polge} for all $k =1,2,\ldots,m-1$, and that $G_{2k-1}^{(r)}(N_r)$ is of the form \eqref{polgo} for $k =m$, implies that $G_{2m}^{(r)}(N_r) = \sum_{j=0}^{m-1} g_{2m,j}^{(r)} \, N_r^{2j+1}$, where
\vspace{2mm}
\begin{subnumcases}{g_{2m,j}^{(r)} =}
    \displaystyle \, \frac{2m}{2m+r} \, g_{2m-1,m-1}^{(r)}, \,\,\, \text{for}
     \,\, j = m-1; \label{case4} \\[1mm]
    \displaystyle \, \frac{1}{2m+r} \bigg[ 2m g_{2m-1,j}^{(r)}
     -r \sum_{k=j+1}^{m-1} \binom{2m}{2k} B_{2m-2k} \, g_{2k,j}^{(r)} \bigg] , \,\,\, \text{for}  \,\, j = 0,1,\ldots, m-2.  \label{case5}
\end{subnumcases}
Likewise, by the induction hypothesis, the coefficient $g_{2m-1,m-1}^{(r)}$ in \eqref{case4} is a positive (rational) number. Furthermore, by invoking the induction hypothesis and Property 2, one can easily deduce that the overall expression in \eqref{case5} is a non-zero (rational) number with sign $(-1)^{m-1-j}$. This means that, for each $j=0,1,\ldots,m-1$, the coefficients $g_{2m,j}^{(r)}$ conform to the required properties, and, consequently, the polynomial $G_{2m}^{(r)}(N_r)$ is of the form \eqref{polge} for all integers $r \geq 1$ and $m \geq 1$.
\end{proof}

Equations \eqref{case1}-\eqref{case3} [respectively, \eqref{case4}-\eqref{case5}] provide us with a concrete procedure to obtain $G_{2m-1}^{(r)}(N_r)$ [$G_{2m}^{(r)}(N_r)$] from the lower-degree polynomials $G_{1}^{(r)}(N_r), G_{3}^{(r)}(N_r), \ldots, G_{2m-3}^{(r)}(N_r)$, and $G_{2m-2}^{(r)}(N_r)$ [$G_{2}^{(r)}(N_r), G_{4}^{(r)}(N_r), \ldots, G_{2m-2}^{(r)}(N_r)$, and $G_{2m-1}^{(r)}(N_r)$]. For example, for $m=5$, equations \eqref{case1}-\eqref{case3} give
\begin{align*}
g_{9,0}^{(r)} & = \frac{3}{19}g_{1,0}^{(r)} - \frac{20}{19}g_{3,0}^{(r)} + \frac{42}{19}g_{5,0}^{(r)}
- \frac{60}{19}g_{7,0}^{(r)},  \\[1mm]
g_{9,1}^{(r)} & = \frac{9}{19}g_{8,0}^{(r)} - \frac{20}{19}g_{3,1}^{(r)} + \frac{42}{19}g_{5,1}^{(r)}
- \frac{60}{19}g_{7,1}^{(r)},  \\[1mm]
g_{9,2}^{(r)} & = \frac{9}{19}g_{8,1}^{(r)} + \frac{42}{19}g_{5,2}^{(r)} - \frac{60}{19}g_{7,2}^{(r)}, \\[1mm]
g_{9,3}^{(r)} & = \frac{9}{19}g_{8,2}^{(r)} - \frac{60}{19}g_{7,3}^{(r)}, \\[1mm]
g_{9,4}^{(r)} & = \frac{9}{19}g_{8,3}^{(r)},
\end{align*}
from which we can get $G_9^{(r)}(N_r)$ if we know the coefficients of the polynomials $G_1^{(r)}(N_r)$, $G_3^{(r)}(N_r)$, $G_5^{(r)}(N_r)$, $G_7^{(r)}(N_r)$, and $G_8^{(r)}(N_r)$.

\begin{remark}
For any integer $r\geq 1$, the binomial coefficient $\binom{n+r}{r+1}$ admits the following representation involving the (unsigned) Stirling numbers of the first kind $\genfrac{[}{]}{0pt}{}{r}{j}$ and the power sum polynomials $S_j(n)$ for each $j=1,2,\ldots,r$, namely (see, e.g., \cite[Equation (8)]{haggard} and \cite[Equation (10)]{cere3}):
\begin{equation*}
\binom{n+r}{r+1} = \frac{1}{r!} \sum_{j=1}^{r}  \genfrac{[}{]}{0pt}{}{r}{j} S_j(n).
\end{equation*}
Therefore, from \eqref{polg} and Theorem \ref{th:5}, it follows that, for all integers $r \geq 1$ and $m\geq 1$, the hyper-sum polynomial $S_m^{(r)}(n)$ can be expressed in the form
\begin{align*}
r! S_{2m-1}^{(r)}(n) & = \Bigg(  \sum_{j=1}^{r}  \genfrac{[}{]}{0pt}{}{r}{j} S_j(n) \Bigg) \times
\Bigg( \sum_{j=0}^{m-1} g_{2m-1,j}^{(r)} \,\Big( n+ \frac{r}{2}\Big)^{2j} \Bigg), \\[-3mm]
\intertext{for odd powers $2m-1$, and}
r! S_{2m}^{(r)}(n) & = \Bigg(  \sum_{j=1}^{r}  \genfrac{[}{]}{0pt}{}{r}{j} S_j(n) \Bigg) \times
\Bigg(  \sum_{j=0}^{m-1} g_{2m,j}^{(r)} \, \Big( n+\frac{r}{2}\Big)^{2j+1} \Bigg),
\end{align*}
for even powers $2m$. It is left as an open problem to determine an explicit formula for the coefficients $g_{2m-1,j}^{(r)}$ and $g_{2m,j}^{(r)}$.
\end{remark}

\section{Conclusion}

Summarizing, in this paper we have derived a new determinantal formula for the hyper-sums of powers of integers $S_m^{(r)}(n)$, embodied in equation \eqref{det} of Theorem \ref{th:4}. Central to the derivation of this determinantal formula is the discovery of the recurrence relation \eqref{lm:1}. Moreover, by virtue of Theorem \ref{th:5}, it turns out that, whenever $r \geq 1$ and $m \geq 1$, the determinant of the lower Hessenberg matrix \eqref{det1} is invariably an even or odd polynomial in $N_r$ of degree $m-1$, according as $m$ is odd or even. Consequently, whenever $r \geq 1$ and $m \geq 1$, the said determinantal formula \eqref{det} gives us $S_m^{(r)}(n)$ as $S_1^{(r)}(n)$ times an even or odd polynomial in $n + \frac{r}{2}$ of degree $m-1$, depending on whether $m$ is odd or even.

We conclude the main text with the following additional remarks.

For the special case in which $r=1$, the recurrence \eqref{lm:1} becomes
\begin{equation}\label{myrec}
    (m+1) S_{m}(n) = m \Big( n + \frac{1}{2} \Big) S_{m-1}(n) - \sum_{k=1}^{m-2} \binom{m}{k} B_{m-k} S_{k}(n),
\end{equation}
from which one can deduce the following determinantal formula for the power sum polynomial $S_m(n) = 1^m + 2^m + \cdots + n^m$, namely,
 \begin{align}\label{det2}
  S_m(n) = \, & \frac{(-1)^{m+1}}{(m+1)!} \Big( N^2 - \frac{1}{4} \Big)  \notag \\[2mm]
  & \quad\,\,  \times  \begin{vmatrix}
  -2N & 3 & 0 & \hdots & \hdots & 0 \\[3pt]
  \binom{3}{1}B_2 & -3N & 4 & 0 & \hdots & 0 \\[3pt]
  \binom{4}{1}B_3 & \binom{4}{2}B_2 & -4N & 5 & \ddots & 0  \\[3pt]
  \vdots & \vdots & \ddots  & \ddots & \ddots & 0 \\[3pt]
  \binom{m-1}{1}B_{m-2} & \binom{m-1}{2}B_{m-3} & \binom{m-1}{3}B_{m-4} & \hdots & -(m-1)N & m \\[3pt]
  \binom{m}{1}B_{m-1} & \binom{m}{2}B_{m-2} & \binom{m}{3}B_{m-3} & \hdots & \binom{m}{m-2}B_2 & -mN \\[3pt]
\end{vmatrix},
\end{align}
where $N$ is a shorthand for $n + \frac{1}{2}$. The above determinantal formula applies to any integer $m \geq 1$, with the convention that the displayed determinant of order $m-1$ is equal to $1$ when $m =1$. Incidentally, both the recurrence \eqref{myrec} and the determinantal formula \eqref{det2} were obtained by the author in \cite{cere2}. As an example, for $m=7$ and $m=8$, formula \eqref{det2} yields
\begin{align*}
S_7(n) & =  \frac{N^8}{8} - \frac{7 N^6}{24} + \frac{49 N^4}{192} - \frac{31 N^2}{384} + \frac{17}{2048}, \\
\intertext{and}
S_8(n) & =  \frac{N^9}{9} - \frac{N^7}{3} + \frac{49 N^5}{120} - \frac{31 N^3}{144} + \frac{127 N}{3840},
\end{align*}
respectively. In general, $S_m(n)$ can be expressed as an even or odd polynomial in $n + \frac{1}{2}$ depending on whether $m$ is odd or even \cite{hersh}. Specifically, for any integer $m \geq 1$, we have
\begin{align*}
S_{2m-1}(n) & = \sum_{j=0}^{m} f_{2m-1,j} \Big( n + \frac{1}{2}\Big)^{2j}, \\[-3mm]
\intertext{and}
S_{2m}(n) & = \sum_{j=0}^{m} f_{2m,j} \Big( n + \frac{1}{2}\Big)^{2j+1}.
\end{align*}
Theorem \ref{th:5} ensures that the coefficients $f_{2m-1,j}$ and $f_{2m,j}$, $j=0,1,\ldots,m$, are non-zero (rational) numbers with alternating signs. Explicit formulas for $f_{2m-1,j}$ and $f_{2m,j}$ in terms of the Bernoulli numbers can be found elsewhere.

On the other hand, for integers $r \geq 0$ and $m \geq 1$, one can readily deduce from \eqref{lm:2} the relations
\begin{align}
S_{2m-1}^{(r+1)}(n) & = \frac{1}{2} S_{2m-1}^{(r)}(n) + \frac{1}{2m} \sum_{k=1}^{m} \binom{2m}{2k} B_{2m-2k}
S_{2k}^{(r)}(n), \label{coffey1} \\
\intertext{and}
S_{2m}^{(r+1)}(n) & = \frac{1}{2} S_{2m}^{(r)}(n) + \frac{1}{2m+1} \sum_{k=1}^{m+1} \binom{2m+1}{2k-1} B_{2m+2-2k}
S_{2k-1}^{(r)}(n),  \label{coffey2}
\end{align}
which gives us the $(r+1)$-fold summation $S_{2m-1}^{(r+1)}(n)$ [respectively, $S_{2m}^{(r+1)}(n)$] in terms of the $r$-fold summations $S_{2}^{(r)}(n), S_{4}^{(r)}(n), \ldots, S_{2m}^{(r)}(n)$, and $S_{2m-1}^{(r)}(n)$ [$S_{1}^{(r)}(n), S_{3}^{(r)}(n), \ldots, S_{2m+1}^{(r)}(n)$, and $S_{2m}^{(r)}(n)$]. It is to be mentioned that relations \eqref{coffey1} and \eqref{coffey2} have been previously derived (in a slightly different form) by Coffey and Lettington using a different method (see Equations (1.11) and (1.10) in \cite{coffey}). For example, considering the specific case when $r = m=3$, and employing \eqref{coffey1} in combination with \eqref{det} leads to
\begin{align*}
S_{5}^{(4)}(n) - \frac{1}{2} S_{5}^{(3)}(n) & = \frac{1}{240} n(n+1)(n+2)(n+3)(2n+3) \\
& \quad\qquad\! \times \left[ \frac{5}{126} \Big( n + \frac{3}{2} \Big)^4 - \frac{5}{252} \Big( n + \frac{3}{2} \Big)^2
- \frac{859}{2016} \right],
\end{align*}
which should be compared with the (equivalent) result obtained in \cite{coffey}, namely,
\begin{align*}
S_{5}^{(4)}(n) - \frac{1}{2} S_{5}^{(3)}(n) & = \frac{1}{240} n(n+1)(n+2)(n+3)(2n+3) \\
& \quad\qquad\! \times \left[ \frac{5}{126} n^2 (n+3)^2 + \frac{10}{63} n(n+3) - \frac{17}{63} \right].
\end{align*}

\begin{center}
\section*{Appendix}
\end{center}

Here, we give a proof of the recurrence \eqref{rec} by mathematical induction. For this, let us consider in the first place the series $s_n = \sum_{j=1}^n a_j$ ($n \geq 1$), with the general term $a_j$ being quite arbitrary. Then, the following relation can be easily derived (see, e.g., \cite[Lemma]{sullivan}):
\begin{equation}\label{ap1}
\sum_{j=1}^n j a_j = (n+1) s_n - \sum_{j=1}^n s_j.  \tag{A1}
\end{equation}
For the purpose of our argument, we choose $a_j = S_m^{(r-1)}(j)$ where $m$ and $r$ are taken to be arbitrary but fixed integers with $m \geq 0$ and $r \geq 1$. Hence, in the language of hyper-sums of powers of integers, \eqref{ap1} can be written as
\begin{equation}\label{ap2}
\sum_{j=1}^n j  S_m^{(r-1)}(j) = (n+1) S_m^{(r)}(n) - S_m^{(r+1)}(n).  \tag{A2}
\end{equation}
Note that, for $r=1$, \eqref{ap2} becomes
\begin{equation*}
 S_{m+1}(n) = (n+1) S_m(n) - S_{m}^{(2)}(n),
\end{equation*}
or,
\begin{equation*}
 S_m^{(2)}(n) = (n+1) S_m(n) - S_{m+1}(n),
\end{equation*}
which is just the recurrence \eqref{rec} for $r=1$. In general, for any arbitrary integer $j \geq 1$, we have
\begin{equation*}
 S_m^{(2)}(j) = (j+1) S_m(j) - S_{m+1}(j).
\end{equation*}
Thus, summing each side of this equation from $j=1$ to $n$ yields
\begin{equation*}
S_{m}^{(3)}(n) = \sum_{j=1}^n j S_m(j) + S_m^{(2)}(n) - S_{m+1}^{(2)}(n).
\end{equation*}
When $r=2$, from \eqref{ap2} we obtain
\begin{equation*}
 \sum_{j=1}^n j S_m(j) = (n+1) S_m^{(2)}(n) - S_m^{(3)}(n),
\end{equation*}
and then
\begin{equation*}
  S_m^{(3)}(n) = \frac{n+2}{2} S_m^{(2)}(n) - \frac{1}{2} S_{m+1}^{(2)}(n),
\end{equation*}
which is just the recurrence \eqref{rec} for $r=2$.

Now, as the induction hypothesis, let us assume that, for any given integer $r\geq 2$, the following relationship holds true
\begin{equation*}
  S_m^{(r)}(j) = \frac{j+r-1}{r-1} S_m^{(r-1)}(j) - \frac{1}{r-1} S_{m+1}^{(r-1)}(j),
\end{equation*}
where $j$ stands for any arbitrary integer $j \geq 1$. Hence, summing over the first $n$ integers on both sides of the last equation, we find
\begin{equation*}
  S_m^{(r+1)}(n) = \frac{1}{r-1} \sum_{j=1}^n j S_m^{(r-1)}(j) + S_m^{(r)}(n) - \frac{1}{r-1} S_{m+1}^{(r)}(n).
\end{equation*}
Finally, by replacing $ \sum_{j=1}^n j S_m^{(r-1)}(j)$ with \eqref{ap2}, and after a trivial rearrangement, we obtain \eqref{rec}.

\end{document}